\documentclass[numbered,pdflatex]{sn-jnl}

\usepackage{graphicx}%
\usepackage{multirow}%
\usepackage{amsmath,amssymb,amsfonts}%
\usepackage{amsthm}%
\usepackage{mathrsfs}%
\usepackage[title]{appendix}%
\usepackage{xcolor}%
\usepackage{textcomp}%
\usepackage{manyfoot}%
\usepackage{booktabs}%
\usepackage{algorithm}%
\usepackage{algorithmicx}%
\usepackage{algpseudocode}%
\usepackage{listings}%

\usepackage{caption}

\usepackage{scalerel}
\usepackage{tikz}
\usepackage{amsfonts,amscd,amsmath,amsthm,amssymb,amsfonts,latexsym,}
\usepackage[usenames,dvipsnames]{xcolor}
\usepackage{enumerate}
\usepackage{manfnt}
\usepackage{hyperref}
\hypersetup{colorlinks=true, linkcolor=blue, filecolor=blue, urlcolor=blue, citecolor=blue, filecolor=blue }

\newcommand{\ol}{\overline}
\newcommand{\ts}{\textsc}
\def\Sp{\vspace{1.5ex}}

\def\AA{\forall}
\def\imp{\to}
\def\toto{\twoheadrightarrow}

\def\vext{\bar{\v}}

\def\md{\mathrm{md}\,}

\def\vf{\phi}
\def\mo{\vDash}

\def\con{\wedge}

\def\emp{\varnothing}

\def\clC{\mathcal{C}}

\def\clA{\mathcal{A}}
\def\clK{\mathcal{K}}

\def\clP{\mathcal{P}}
\def\clI{\mathcal{I}}
\def\clJ{\mathcal{J}}
\def\clK{\mathcal{K}}

\def\Log{\logicts{Log}}

\newcommand\logicts[1]{\mathrm{#1}}
\newcommand\languagets[1]{\logicts{#1}}

\def\PV{\languagets{PV}}
\def\NLog{\logicts{K}}

\def\mR{\mathbb{R}}

\def\mQ{\mathbb{Q}}

\theoremstyle{definition}
\newtheorem{definition}{Definition}
\newtheorem{remark}{Remark}
\newtheorem{theorem}{Theorem}
\newtheorem*{theorem*}{Theorem}

\newtheorem{lemma}[theorem]{Lemma}
\newtheorem{prp}[theorem]{Proposition}
\newtheorem{proposition}[theorem]{Proposition}

\newtheorem{corollary}{Corollary}
\newtheorem{problem}{Problem}

\newtheorem*{problem*}{Problem}
\newtheorem*{remark*}{Remark}
\newtheorem{claim}{Claim}

\def\EE{\exists}
\def\AA{\forall}

\def\Log{\mathop{\mathrm{Log}}}
\def\Alg{\mathop{\mathrm{Alg}}}

\def\Di{\lozenge}

\def\val{v}
\def\valu{u}
\def\vext{\bar{\val}}

\newcommand\R{\mathbb{R}}
\newcommand\Q{\mathbb{Q}}
\newcommand\N{\mathbb{N}}
\newcommand\Z{\mathbb{Z}}
\newcommand\Dl{\Di_{\scaleto{<}{3.5pt}\,1}}
\newcommand\Dg{\Di_{\scaleto{>}{3.5pt}\,1}}
\newcommand\De{\Di_{\scaleto{=}{2.5pt}\,1}}
\newcommand\Dle{\Di_{\scaleto{\leq}{3.5pt}\,1}}
\newcommand\Logg{\Log_{\scaleto{\,>}{3.5pt}1}}
\newcommand\Logl{\Log_{\scaleto{\,<}{3.5pt}1}}
\newcommand\Loge{\Log_{\scaleto{\,=}{2.5pt}1}}
\newcommand\Logle{\Log_{\scaleto{\,\leq}{3.5pt}1}}

\author{\fnm{Gabriel} \sur{Agnew}}
\author{\fnm{Uzias} \sur{Gutierrez-Hougardy}}
\author{\fnm{John}  \sur{Harding}}
\author{\fnm{Ilya} \sur{Shapirovsky}}
\author{\fnm{Jackson} \sur{West}}

\affil{New Mexico State University}

\abstract{We consider logics derived from Euclidean spaces $\R^n$. Each Euclidean space carries relations consisting of those pairs that are, respectively, distance more than~1 apart, distance less than 1 apart, and distance 1 apart. Each relation gives a uni-modal logic of $\R^n$ called the farness, nearness, and constant distance logics, respectively. These modalities are expressive enough to capture various aspects of the geometry of $\R^n$ related to bodies of constant width and packing problems. This allows us to show that the farness logics of the spaces $\R^n$ are all distinct, as are the nearness logics, and the constant distance logics. The farness and nearness logics of $\R$ are shown to strictly contain those of $\Q$, while their constant distance logics agree. It is shown that the farness logic of the reals is not finitely axiomatizable and does not have the finite model property. }


\title{On distance logics of Euclidean spaces}
\begin{document}
\maketitle
\section{Introduction}

Let $(X,d)$ be a metric space, which we usually refer to as simply $X$. There are natural ways to associate a modal logic to $X$. McKinsey and Tarski considered topological closure as a modal operator. 
Under this modality, the logic of any dense in itself metric space is Lewis' modal logic S4 \cite{MK-TAR44,RasiowaSikorski1963}.
One can also use the metric directly to define binary relations on $X$ and then the consider the associated modal operators. For example, for each real number $r>0$ set $R_{<r}=\{(x,y)\mid d(x,y)<r\}$
and let $\Di_{<r}$ be the associated modal operator on the powerset of $X$. Relations $R_{\leq r}$ etc., as well as their associated modal operators $\Di_{\leq r}$ etc., are defined in analogous ways. 

The general setting of the series of papers \cite{Kutz-Sturm-Suzuki-Wolter-Zakharyaschev2002, Kutz-Sturm-Suzuki-Wolter-Zakharyaschev2003,Wolter2005,KuruczWZ05,Kutz2007}
considers metric spaces with some specified collection of modal operators chosen from the topological closure operator and the operators $\Di_{<r}$ etc. Numerous results were established related to the axiomatizability and decidability of the modal logics of the class of all metric spaces with different families of operators. For instance, the logic of all metric spaces with operators $\Di_{<r}$ for each $r>0$ is given by the family of axioms:

\begin{itemize}[leftmargin=5ex]
\item[] $p\,\to\,\Di_{\scaleto{<}{3.5pt}\,r} p$ \vspace{.5ex}
\item[] $p\,\to\,\Box_{\scaleto{<}{3.5pt}\,r}\Di_{\scaleto{<}{3.5pt}\,r}p$ \vspace{.5ex}
\item[] $\Di_{\scaleto{<} {3.5pt}\,r}\Di_{\scaleto{<} {3.5pt}\,s}\,p\to\,\Di_{\scaleto{<}{3.5pt}\,r+s}\, p$
\end{itemize}
\vspace{1ex}

\noindent The first axiom says that for a subset $P$ of a metric space, we have $P$ is contained in the set of points of distance less than $r$ from $P$; the second expresses the symmetry of the metric; the third comes from the triangle inequality. The validity of these axioms is obvious, but some effort is required for their completeness. A related direction that we do not follow here considers binary modality reflecting where $x$ is closer to $z$ than $y$, see, e.g., \cite{Sher-Wolt-Zakh2010,AutomatedReasoningTableaux2009}. 

In this paper we consider Euclidean spaces $\R^n$ with the standard metric and equipped with a single modal operator $\Di_{>r}, \Di_{<r}$ or $\Di_{=r}$. By the scale invariance of $\R^n$ it is enough to consider the case when $r=1$. By the \emph{farness logic} of $\R^n$ we mean the set of modal formulas valid in $\R^n$ with the $\Dg$ modality, the \emph{nearness logic} of $\R^n$ is the set of formulas valid in the $\Dl$ modality, and the \emph{constant distance} logic is the set of formulas valid in the $\De$ modality. We denote these as 
\begin{align*}
\mbox{$\Logg(\R^n)$}\,\,&=\,\,\mbox{ the formulas valid in $\R^n$ in the $\Dg$ modality}\\[.5ex]
\mbox{$\Logl(\R^n)$}\,\,&=\,\,\mbox{ the formulas valid in $\R^n$ in the $\Dl$ modality}\\[.5ex]
\mbox{$\Loge(\R^n)$}\,\,&=\,\,\mbox{ the formulas valid in $\R^n$ in the $\De$ modality}
\end{align*}

Under the topological closure modality, the logic of $\R^n$ for $n\geq 1$ is S4 since each is a dense in itself metric space. The situation for farness, nearness, and constant distance logics is much different. Each modality is sufficiently expressive to capture geometric features of $\R^n$ specific to its dimension. These features range from properties of convex sets, to bodies of constant width, to sphere packing problems, and chromatic numbers of unit distance graphs.

Using these techniques, we show that the farness logics $\Logg(\R^n)$ are pairwise distinct; that the nearness logics $\Logl(\R^n)$ are pairwise distinct and form a strictly decreasing chain; and that the constant distance logics $\Loge(\R^n)$ are pairwise distinct and the poset of these logics contains an infinite antichain. 

We compare the farness, nearness, and constant distance logics of $\R$ to those of the rationals $\Q$. In each case the logic of $\R$ is shown to contain the corresponding logic of $\Q$. For the constant distance modality we have equality $\Loge(\R)=\Loge(\Q)$. 
It is perhaps surprising that the farness and nearness modalities are each sufficiently expressive to allow one to formulate versions of connectivity to 
%
show the nearness and farness logics of $\R$ are distinct from those of $\Q$. 
The farness logic of $\R$ is analyzed in more detail. It is shown that it cannot be axiomatized using only finitely many variables, and that it does not have the finite model property.

The paper is arranged in the following way. The second section is preliminaries. In the third section Theorem~\ref{thm: one non-containment} provides that the farness logics of $\R^n$ for $n\in\N$ are distinct; Proposition~\ref{prp: max in ball} provides the same for the nearness logics; and Corollary~\ref{cor:equidist-antichain} does the same for the constant distance logics. In the fourth section we compare the farness, nearness, and constant distance logics of $\R$ and $\Q$;
these results are given in 
Proposition \ref{prop:dim1:r=1}, Theorems \ref{thm:combined} and  \ref{thm:less:Qincluded-in-R}. In the fifth section we discuss axiomatizability and the finite model property for $\R$; the main results of this section are Proposition~\ref{prp:nonFA} and Theorem~\ref{thm:noFMP}. In the sixth section we give a summary of the results and several open problems as well as suggested directions for further study.

\section{Preliminaries}\label{sec:prel} 

In this section, we provide the basic notation and terminology used throughout the paper. It is assumed that the reader is familiar with the basics of modal logic, as described in \cite{BDV}.

\subsection{Modal syntax} 
We consider uni-modal logics with a single modality denoted by $\Dg$, $\Dl$, $\De$, depending on context, and usually written in a given context simply $\Di$. 
The set of {\em modal formulas} is built from a countable set of {\em variables} $\PV=\{p_0,p_1,\ldots\}$ using Boolean connectives $\bot,\imp$ and the unary connective $\Di$. Other Boolean connectives are defined as abbreviations in the standard way, and $\Box\vf$ denotes $\neg \Di \neg \vf$.

\subsection{Relational semantics} 
A {\em frame} $F=(X,R)$ consists of a set $X$ and a binary relation $R$ on $X$. We write $x\,R\,y$ to indicate that $x$ and $y$ are related and for $x\in X$ and $Y\subseteq X$, we put $R(x)=\{y\mid xR y\}$ and we write $R[Y]$ for the relational image $\{x\mid y\, R\, x$ for some $y\in Y\}$.

A {\em valuation} in a frame $F$ is a map $v:\PV\to \clP(X)$ from the set of propositional variables to the powerset of $X$. A {\em model on} $F$ is a pair $(F,\val)$, where $\val$ is a valuation. \emph{Truth} at a point $x$ of a model is defined in the standard way, in particular $F,x\mo_v\Di \vf$ if  ${F,y\mo_v\vf}$  for some $x\,R\,y$,  and thus $F,x\models_v\Box\vf$  if $F,y\models_v\vf$ for all $y$ with $x\,R\, y$. Set $$\vext(\vf)\,=\,\{x\mid F,x\mo_v\vf\}.$$
A  formula $\vf$ is {\em true} in a model $(F,v)$, in symbols $F \mo_v\vf$,  if $F,x\mo_v\vf$ for all $x$ in $X$, that is, if $\ol{v}(\vf)=X$. A formula $\vf$ is {\em valid} in a frame $F$, in symbols $F\mo\vf$, if $\vf$ is true for every valuation in $F$. A formula $\vf$ is \emph{satisfiable} in $F$ if $\neg\vf$ is not valid in $F$, i.e., if there is a valuation $v$ in $F$ and a point $x$ in $X$ with $F,x\models_v\vf$.

\subsection{Algebraic  semantics} 
A {\em modal algebra} $B$ is a Boolean algebra endowed with a unary operation $\Di$ 
that distributes over finite joins, so in particular satisfies $\Di(0)=0$. An \emph{interpretation} in a modal algebra is a mapping $v:\PV\to B$ of the propositional variables into its underlying Boolean algebra. An interpretation extends in the usual way to a mapping $\ol{v}$ from the collection of all modal formalas to $B$. A modal formula $\vf$ is \emph{true} under the interpretation $v$ if $\ol{v}(\vf)=1$ and is \emph{valid} in $B$ if it is true for all interpretations, i.e., if~$B$ satisfies the equation $\vf=1$. 


The {\em algebra of a frame}  $F=(X,R)$, written $\Alg{F}$, is the powerset Boolean algebra of $X$ with the unary operation $\Di$ given by  $\Di(Y)=R^{-1}[Y]$ for each subset $Y\subseteq X$. It is immediate that  the algebra of a frame is a modal algebra, that valuations $v$ in $F$ correspond to interpretations in $\Alg{F}$, and that the value of $\ol{v}$ is the same in either context. See \cite[Section 5.2]{BDV} for details.


A formula is {\em valid in a class} $\clC$ of (relational or algebraic) structures if it is  valid in every structure in $\clC$. Validity of a set of formulas means validity of every formula in this set.

\subsection{Normal logics} 
A ({\em propositional normal modal}) {\em logic} is a set $L$ of formulas that contains all classical tautologies, the axioms $\neg\Di \bot$ and $\Di  (p \vee q) \imp \Di  p \vee \Di  q$, and is closed under the rules of modus ponens, substitution, and the following rule of {\em monotonicity}: $\vf\imp\psi\in L$ implies $\Di \vf\imp\Di \psi\in L$. The smallest logic is denoted $\NLog$. For a logic $L$ and a set $\Phi$ of modal formulas, $L+\Phi$ is the smallest normal logic that contains $L\cup\Phi$.

An {\em $L$-structure}, a frame or algebra, is a structure where each formula in $L$ is valid. 
The set of all formulas that are valid in a class $\clC$ of relational or algebraic structures is a propositional normal modal logic called the {\em logic of} $\clC$ and denoted $\Log{\clC}$. When $\clC$ consists of a single frame $F$ or modal algebra $B$, this is written $\Log{F}$ or $\Log{B}$. Any propositional normal modal logic $L$ is the logic of the class of all $L$-algebras that satisfy all equations $\vf=1$ for $\vf\in L$, see, e.g., \cite{CZ}. A logic $L$ is {\em Kripke complete} if $L$ is the logic of a class of frames. 
A logic has the {\em finite frame property}, if it is the logic of a class of finite frames, meaning frames whose underlying sets are finite. 
In our setting, this is equivalent to the more common notion of the {\em finite model property}, see, e.g., \cite[Section 3.4]{BDV}.

\section{Distinguishing logics of $\mR^n$}

In this section we consider the logics of the Euclidean spaces $\R^n$ for $n\in\N$ in languages with a single modality $\Dg$, $\Dl$, or $\De$. In each signature we show that the language is sufficiently expressive to distinguish each of the logics $\Log{\R^n}$. We additionally show that for the $\De$ signature we show there is an infinite anti-chain among the logics for the $\R^n$ and for the $\Dl$ signature the logics form a strictly decreasing chain. 

To provide an informal overview, for the farness logics considered in Section~3.1 we employ Helly's theorem from convex geometry to produce a formula valid in small dimension but not in large dimension. We also provide another approach to the same end using volumes of unit bodies of constant width. 

In Section~3.2 we consider nearness logics. There is an obvious $p$-morphism from larger to smaller dimension providing a containment. To show this containment is strict we use a packing argument to provide a formula valid in small dimension but not in large dimension, roughly, that a certain number of balls of radius 1/2 in a ball of radius 1 must overlap. 

In Section~3.3 we show that the logics of fixed distance contain an infinite anti-chain. We first note that in dimension $>2$ we have $\De\circ\De = \Di_{\scaleto{\leq}{3.5pt}\,2}$, giving access to results obtained for nearness logics and formulas valid in small dimension but not in large dimension. To obtain formulas valid in large dimension and not in small dimension we use coloring properties of unit distance graphs.

\subsection{Farness logics}
The \emph{farness logic} has the single modality~$\Dg$. Since this sub-section will treat only this modality, we write it simply as $\Di$, and for a metric space $X$ write $x\,R\,y$ iff $d(x,y)>1$. When $x\,R\,y$ we say that $x,y$ are \emph{far}. We write $\Logg(X)$ for the modal formulas with this modality that are valid in $X$.  

A first observation is that in any unbounded metric space, and in particular in every Euclidean space, $R\circ R$ is the universal relation. So if $v$ is a valuation, then writing $\Di^2$ for $\Di\Di$ and $\Box^2$ for $\Box\Box$, we have
\begin{equation}
\ol{v}(\Di^2p)\,\,=\,\,\begin{cases}
\R^n&\mbox{if }v(p)\neq\emptyset\\ \emptyset&\mbox{otherwise}    
\end{cases}
\quad \mbox{ and }\quad
\ol{v}(\Box^2p)\,\,=\,\,\begin{cases}
\R^n&\mbox{if }v(p)=\R^n\\ \emptyset&\mbox{otherwise}    
\end{cases}
\label{eqn: universal in Dg}
\end{equation}

Another basic observation that we will use involves an \emph{anti-clique}, a set of points with no pair of them far, i.e. each pair $x,y$ of distance at most 1 apart. A maximal anti-clique is then an anti-clique that is not properly contained in any other. Observe that 
\begin{equation}
\R^n\models_v \Box^2(p\leftrightarrow \neg\Di p) \,\mbox{ iff }\, v(p)\mbox{ is a maximal anti-clique}
\label{eqn: max anti-clique}
\end{equation}
It is easily seen that a closed ball of radius $\frac{1}{2}$ is a maximal anti-clique in any  $\R^n$. In dimension 1, it is not difficult to see that the closed intervals of length 1 are exactly the maximal anti-cliques. In higher dimensions, there are others, for instance, the Reuleaux triangle shown 
in Figure \ref{fig:rel} 
is a maximal anti-clique in $\R^2$. In fact, the maximal anti-cliques in the farness logic are the closed convex bodies of constant width 1. We will establish this fact later in this sub-section when it is needed. 
\vspace{1ex}

\begin{figure}[h]
\centering
\begin{tikzpicture}
\draw (2,0) arc[start angle=0, end angle=60,radius=2cm];
\draw (1,1.733) arc [start angle=120, end angle=180, radius=2cm];
\draw (0,0) arc [start angle=240, end angle=300, radius=2cm];
\end{tikzpicture}
\captionsetup{justification=centering}
\caption{Reuleaux triangle}
\label{fig:rel}
\end{figure}

\begin{prp}
\label{prp: R>1 not contained in R^n}
$\Logg(\R)$ is not contained in $\Logg(\R^n)$ for any $n>1$. 
\end{prp}

\begin{proof}
The idea is simple. In $\R$, if $A, B$ and $C$ are closed balls of radius $\frac{1}{2}$, i.e. closed unit intervals, such that any two intersect non-trivially, then the three intersect non-trivially. However, the corresponding statement is not true in $\R^n$ for $n>1$. Thus 
\[
\bigwedge_{i=1}^3\Box^2 (p_i\leftrightarrow \neg \Di p_i)\wedge \bigwedge_{i=1}^3 \Di^2 (\bigwedge_{i\neq j}p_j)\,\,\to\,\, \Di^2(p_1\wedge p_2\wedge p_3)
\]
\vspace{.5ex}

\noindent is valid in $\R$ but not in $\R^n$ for $n>1$. 
\end{proof}

We next provide an extension of Proposition~\ref{prp: R>1 not contained in R^n}. For this, we need to extend the geometric reasoning used in that proposition. Recall than a regular unit $n$-simplex consists of $n+1$ points in $\R^k$ for some $k\geq n$ such that any two are distance 1 apart. It is well known that the unit $n$-simplex can be realized in $\R^{n+1}$ by the scaled standard basis vectors $\frac{1}{\sqrt{2}}e_1,\ldots,\frac{1}{\sqrt{2}}e_{n+1}$. Using this, one can determine that the height $h_n$ of a unit $n$-simplex, 
i.e., the distance from a vertex to an opposing face, 
and distance $d_n$ from a vertex to its centroid are given by 
%
\[
h_n\,=\,\sqrt{\frac{n+1}{2n}}\quad\mbox{and}\quad
d_n\,=\,\sqrt{\frac{n}{2(n+1)}}
\]

\noindent Any $n$ vertices of a regular unit $n$-simplex lie in a regular unit $(n-1)$-simplex. So there is exactly one point within distance $d_{n-1}$ of each set of $n$ vertices. However, the remaining vertex is height $h_n$ above the hyperplane of the others, so there are no points within distance $d_{n-1}$ of all of the vertices. Since the regular unit $n$-simplex can be embedded into any $\R^k$ for $k\geq n$, we have the following
\Sp

\begin{lemma}
\label{lem: exist balls}
For each $n,k\in\N$, with $n\leq k$, there are $n+1$ closed balls of equal radius in $\R^k$ such that any $n$ of these balls contain a common point of intersection and the intersection of all $n+1$ of the balls is empty.    \end{lemma}
\Sp

On a related note, Helly's theorem, see e.g. \cite[p.~48]{Lay-convex}, says that if, in $\R^n$, we have a finite collection of $k\geq n+1$ convex sets with any $n+1$ of them containing a common point of intersection, then the intersection of all of them is non-empty. 
\Sp

\begin{theorem} 
\label{thm: one non-containment}
If $m>n$ then $\Logg(\R^n)$ is not contained in $\Logg(\R^m)$.  
Consequently, the logics $\Logg(\R^n)$ are all distinct.
\end{theorem}

\begin{proof}
As we observed earlier, for any dimension $k\in\N$, a closed ball of radius $\frac{1}{2}$ is a maximal anti-clique. If $k>1$, not every maximal anti-clique is a closed ball, but no matter, each maximal anti-clique is a convex set since the convex closure of an anti-clique is again such. So for any $n\in \N$, Helly's theorem says that a collection of $n+2$ maximal anti-cliques with the property that any $n+1$ of them have non-empty intersection has the property that all $n+2$ have non-empty intersection. Thus the following formula is valid in $\R^n$.
\[
\bigwedge_{i=1}^{n+2}\Box^2 (p_i\leftrightarrow \neg \Di p_i)\wedge \bigwedge_{i=1}^{n+2} \Di^2 (\bigwedge_{i\neq j}p_j)\,\,\to\,\, \Di^2(\bigwedge_{i=1}^{n+2}p_i)
\]

\vspace{-.5ex}

\noindent However, by Lemma~\ref{lem: exist balls}, for each $k\geq n+1$ there are $n+2$ closed balls of equal radius, which can obviously be chosen to be $\frac{1}{2}$, such that any $n+1$ of then contain a common point of intersection, but the intersection of all $n+2$ balls is empty. Since closed balls of radius $\frac{1}{2}$ are maximal anti-cliques, this formula fails in $\R^k$ for each $k>n$.  
\end{proof}

 
This theorem illustrates the expressive power of the language of farness in allowing us to make contact with Helly's theorem. We next illustrate an additional instance of the expressive power of this language related to certain subsets of $\R^n$ known as ``bodies of constant width'' that we now describe. 

A \emph{hyperplane} $H$ in $\R^n$ is a translation $v+U$ of a linear subspace $U$ of codimension 1. In this case, a unit vector that is a orthogonal to $U$  is called the \emph{direction} of $H$. 
Each hyperplane separates the space into two closed half-spaces. See e.g. \cite{Lay-convex}.
Suppose $K$ is a compact, convex set, or a \emph{body}. A hyperplane $H$ is a \emph{supporting hyperplane} of $K$ if $K$ lies entirely in one closed half-plane determined by $H$ and $H$ contains at least one boundary point of $K$. If $K$ has non-empty interior, then for each unit vector $u$, there are two distinct supporting hyperplanes of $K$ that have direction $u$, and the distance between these hyperplanes is called the \emph{width} of $K$ in direction~$u$. A body has constant width if the width in each direction is the same. By a \emph{unit body of constant width}, or \ts{ubcw}, we mean a compact convext set of constant width 1.  Clearly any closed  $n$-ball in $\R^n$ of radius $\frac{1}{2}$ is a \ts{ubcw}, but there are many others, such as the Reuleaux triangle in $\R^2$. For details see e.g. \cite{Martini-BCW}.
\Sp

\begin{proposition}
Maximal anti-cliques in $\R^n$ are exactly the \ts{ubcw}s.     
\end{proposition}

\begin{proof}
Note that if $M$ is a maximal anti-clique, then it is easily seen that $M$ is convex. It is also closed and bounded, hence compact. We make use of two facts. First,  a compact convex set $K$ is a \ts{ubcw} iff $K-K=\{x-y\mid x,y\in K\}$ is a closed ball of radius 1 \cite[Thm.~1]{Chakerian}. 
For the second fact, recall that the diameter of a bounded set is the supremum of the distances between points in it. Thus, the diameter of a body of constant width is its width. The second fact that we use is that any bounded set of diameter $d$ is a subset of a body of constant width with diameter $d$, and hence also of a body of constant width with any chosen diameter $\geq d$
\cite[Thm.~11.4]{conv-discr}.
Surely a maximal anti-clique has diameter at most 1. So, our result follows if we show that every \ts{ubcw} $K$ is a maximal anti-clique. Since $K-K$ is a ball of radius 1, we have that $\|x-y\|\leq 1$ for each $x,y\in K$, so $K$ is an anti-clique. Suppose $z\not\in K$. A basic fact of convex geometry is that there is a supporting hyperplane of $K$ with $K$ in one of its half-planes and $z$ in the other \cite{Lay-convex}. 
It follows that $z$ is distance $>1$ from some point in $K$, thus $K$ is a maximal anti-clique. 
\end{proof}

Call a family $\mathcal{M}=M,M_1,\ldots,M_k$ of \ts{ubcw}s in $\R^n$ a special family if $M_1,\ldots,M_k$ are pairwise disjoint and all intersect $M$ non-trivially, and call $k$ the size of the special family $\mathcal{M}$. The union of a special family has diameter at most 3 and it is known that there is a lower bound on the volume of a \ts{ubcw} in $\R^n$. Thus, there is a maximum value $\sigma(n)$ of the size of a special family in $\R^n$. 

Consider for each $k\in\N$ the following formulas in the $\Dg$ modality

\[
\psi_k\,=\, (p\leftrightarrow \neg\Di p)\, \wedge\bigwedge_{i=1}^k (p_i\leftrightarrow\neg\Di p_i)
\quad\mbox{and}\quad
\chi_k\,=\, \bigwedge_{i=1}^k\Di^2(p\wedge p_i)
\]
\vspace{1ex}

\noindent Suppose $v$ is a valuation for $\R^n$ and let $v(p)=M$ and $v(p_i)=M_i$ for $i\leq k$. By item~(\ref{eqn: universal in Dg}) on page~\pageref{eqn: universal in Dg} we have that $\ol{v}(\Box^2\psi_n)$ and $\ol{v}(\chi_n)$ are either empty or all of $\R^n$. By item~(\ref{eqn: max anti-clique}) on page~\pageref{eqn: max anti-clique} we have $\ol{v}(\Box^2\psi_k)=\R^n$ iff $M$ and each $M_1,\ldots,M_k$ are \ts{ubcw}s, and $\ol{v}(\chi_k)=\R^n$ iff each $M_i$ intersects $M$ non-trivially. So, if $k>\sigma(n)$, the \ts{ubcw} $M_1,\ldots,M_k$ cannot be pairwise disjoint, so the following formula $\phi_k$ is valid in $\R^n$

\[
\phi_k\,=\, \Box^2\psi_k\wedge\chi_k\,\to \bigvee_{i\neq j}\Di(p_i\wedge p_j)
\]
\vspace{1ex}

\noindent Conversely, if $k\leq\sigma(n)$, there is a special family $M,M_1,\ldots,M_k$ in $\R^n$, and setting $v(p)=M$ and $v(p_i)=M_i$ for $i\leq k$ provides a valuation for $\R^n$ for which the premise of $\phi_k$ holds but the conclusion does not. Thus, we have the following
\Sp

\begin{prp}
\label{prp: stugg}
$\R^n\models\phi_k$ iff $\sigma(n)<k$.    
\end{prp}
\Sp

A particular instance of a special family is the case of a unit sphere $M$ in $\R^n$ and pairwise disjoint unit spheres $M_1,\ldots,M_k$ that touch $M$ but do not overlap it. The maximum number of such spheres touching $M$ is known as the \emph{kissing number} $\kappa(n)$ for spheres in $\R^n$. Clearly $\kappa(n)\leq\sigma(n)$. Thus, our language allows us to provide a bound on the kissing number. 
It is known that  $\kappa(n)$ is an increasing sequence 
converging to $\inf$ \cite{Shannon} (but it is unknown if it is strictly increasing). This is another geometric route to show
that there are infinitely many distinct logics $\Logg(\R^n)$, a weaker version of Theorem \ref{thm: one non-containment}.

\subsection{Nearness logics} The \emph{nearness logic} has the single modality $\Dl$. Since this sub-section will treat only this modality, we write it simply as $\Di$, and write $x\,R\,y$ iff $d(x,y)<1$. We say $x,y$ are \emph{close} if $x\,R\,y$. For a metric space $X$ we write $\Logl(X)$ for the modal formulas with this modality that are valid in $X$. 

A \emph{weak anti-clique} in $\R^n$ is a set of points such that no distinct pair of them is close, thus any two distinct elements are distance at least 1 apart. Weak anti-cliques can be infinite, but by a simple volume argument, any weak anti-clique that is contained in an open ball of radius 1 in $\R^n$ must be finite. 
\Sp

\begin{definition}
Let $M_n$ be the maximal cardinality of a weak anti-clique that is contained in an open ball of radius 1 in $\R^n$.   
\end{definition}
\Sp

Let $X$ be a weak anti-clique of cardinality $M_n$ that is contained in the open ball of radius 1 centered at the origin in $\R^n$. Then the origin does not belong to~$X$ since that would imply that the origin is its only element and there are such weak anti-cliques with 2 elements. So there is an $\epsilon$-neighborhood of the origin that does not contain any points of $X$. Embedding $X$ into the hyperplane of $\R^{n+1}$ spanned by the first $n$ standard basis vectors gives a weak anti-clique of $\R^{n+1}$ that is contained in the unit ball centered at the origin. We can add a point $(0,\ldots,0,1-\lambda)$, for some $\lambda>0$, and get a strictly larger weak anti-clique that is contained in the open unit ball centered at the origin of $\R^{n+1}$. Thus $1<M_n<M_{n+1}$ for each $n\geq 1$. 
\Sp

\begin{proposition}
\label{prp: max in ball}
$\Logl(\R^n)$ strictly contains $\Logl(\R^{n+1})$ for each $n\in\N$.    
\end{proposition}

\begin{proof}
Let $f:\R^{n+1}\to \R^n$ take $(x_1,\ldots,x_{n+1})$ to $(x_1,\ldots,x_n)$. It is easily seen that this is a surjective p-morphism, giving the containment desired. To show that this containment is strict, consider the following formulas in the variables $p_1,\ldots,p_k$ where $k=M_{n+1}$. 
\[
\psi_i\, = \, p_i\wedge\bigwedge_{j\neq i}\neg\Di\,p_j\quad\mbox{ and }\quad \phi\,=\, \neg\bigwedge_{i=1}^k\Di\psi_i
\]

We first show that $\phi$ is not valid in $\R^{n+1}$. Let $x_1,\ldots,x_k$ be a weak anti-clique in the open unit ball of $\R^{n+1}$ centered at the origin $o$. Choose a valuation $v$ so that for each $i\leq k$ we have  $v(p_i)=\{x_i\}$. Since $x_i$ is distance at least 1 from $x_j$ for $j\neq i$, we have that $x_i$ belongs to $\ol{v}(\psi_i)$, and as $x_i$ is in the open unit ball centered at the origin, we have that $o$ belongs to $\ol{v}(\Di\psi_i)$ for each $i\leq k$. Thus $\psi$ is not valid. 

To see that $\phi$ is valid in $\R^n$, we must show that for any valuation $v$, the set $\bigcap\{\,\ol{v}(\Di\psi_i)\mid i\leq k\}$ is empty. Suppose to the contrary that some $x$ belongs to this set. Then for each $i\leq k$ there is an element $x_i$ in $\ol{v}(\psi_i)$ that is close to $x$. Then for $i\neq j$ we have $x_i\in\ol{v}(p_i)$, $x_j\in\ol{v}(p_j)$, and $x_i\not\in\ol{v}(\Di p_j)$. So $x_i$ is distance at least 1 from everything in $\ol{v}(p_j)$, and in particular, $x_i$ is distance at least 1 from $x_j$. Thus $x_1,\ldots,x_k$ is a weak anti-clique in the open ball of radius 1 centered at $x$. But $k=M_{n+1}>M_n$, contrary to the definition of $M_n$. Thus $\phi$ is valid in $\R^n$. 
\end{proof}

\begin{remark}
\label{rmk: <=1}
Before concluding this sub-section, we make an observation that will be used in the following subsection. Suppose that for the preceding proof, we were considering the modality $\Dle$ rather than $\Dl$. We could let $\ol{M}_n$ be the maximal cardinality of a weak anti-clique that is contained in a closed unit ball in $\R^n$. Here a weak anti-clique is a collection of points $x_1,\ldots,x_n$ such that $d(x_i,x_j)>1$ for $i\neq j$. The argument given above shows that each $\ol{M}_n$ is finite and $1<\ol{M}_n<\ol{M}_{n+1}$. Let $k=\ol{M}_{n+1}$ and consider the same formulas $\psi_i$ and $\phi$ as in the proof of Proposition~\ref{prp: max in ball}. Then, with obvious modifications, the proof of Proposition~\ref{prp: max in ball} shows that $\phi$ is valid in $\Logle(\R^n)$ and fails in $\Logle(\R^{m})$ for each $m>n$.  
\end{remark}

\subsection{Logics of fixed distance} The \emph{fixed distance logic} has $\De$ as its single modality. Since this sub-section will treat exclusively this modality, we write it simply as $\Di$, and write $x\,R\,y$ iff $d(x,y)=1$. For a metric space $X$ we write $\Loge(X)$ for the modal formulas with this modality that are valid in $X$. We begin with the following observation.
\Sp

\begin{lemma}
\label{lem: dist 2}
In $\R^n$ for $n>1$ we have that $\Di^2$ is the operator $\Di_{\scaleto{\leq}{3.5pt}\,2}$. 
\end{lemma}

\begin{proof}
If $x,y\in\R^n$ with $d(x,y)\leq 2$, the the balls of radius 1 centered at $x$ and $y$ have a point $z$ in common. It follows that $R^2=\{(x,y)\mid d(x,y)\leq 2\}$.      
\end{proof}

\begin{prp}
$\Loge(\R)$ is incomparable to each $\Loge(\R^n)$ for $n>1$.     
\end{prp}

\begin{proof}
We first show that $\R^n$ satisfies the following formula iff $n=1$.  
\[
\bigwedge^3_{i=1} \Di p_i\,\to\, \bigvee_{i\neq j}\Di(p_i\wedge p_j)
\]
Suppose $v$ is a valuation for $\R$ and $x$ belongs to $\ol{v}(\Di\, p_i)$ for $i=1,2,3$. Then one of $x+1,x-1$ must belong to each $v(p_i)$, hence one of these belongs to $v(p_i),v(p_j)$ for some distinct $i,j$. It follows that this formula is valid in $\R$. In $\R^n$ for $n\geq 2$ we can find distinct points $x_1,x_2,x_3$ all distance 1 from the origin $o$. Choosing a valuation with $v(p_i)=\{x_i\}$ for $i=1,2,3$ yields that the origin belongs to $\ol{v}(\Di\,p_i)$ for each $i$, but  $\ol{v}(\Di(p_i\wedge p_j))$ is empty for $i\neq j$. So this formula is not valid in $\R^n$ for $n>1$. 

Next, consider the formula 
\[
\Di p\to\Di\Di p
\]

\noindent By Lemma~\ref{lem: dist 2}, this is valid in $\R^n$ for $n>1$, and is clearly not valid in $\R$. 
\end{proof}

In Remark~\ref{rmk: <=1} we observed that there is a formula in the $\Di_{\scaleto{\leq}{3.5pt}\,1}$ modality that is valid in $\R^n$ but not in any $\R^m$ for $m>n$. Clearly the same is true if we replace the modality $\Di_{\scaleto{\leq}{3.5pt}\,1}$ with $\Di_{\scaleto{\leq}{3.5pt}\,2}$. Thus, by Lemma~\ref{lem: dist 2}, we have the following:
\Sp

\begin{prp}
$\Loge(\R^n)$ is not contained in $\Loge(\R^m)$ for $2\leq n<m$.     
\end{prp}
\Sp

The previous propositions show that the logics $\Loge(\R^n)$ are all distinct. We suspect that they form an antichain. For this, we need to show $\Loge(\R^m)$ is not contained in $\Loge(\R^n)$ for $2\leq n<m$. We provide the result for $n=2$.
\Sp

\begin{prp}
$\Loge(\R^m)$ is not contained in $\Loge(\R^2)$ for $2<m$.  
\end{prp}

\begin{proof}
Consider the formula
\[
p_1 \wedge \Di p_2 \,\,\to\,\, \Di(\Di p_1 \wedge \Di p_2\wedge \Di(\Di p_1 \wedge \Di p_2))
\]
\vspace{-2ex}

\noindent This fails in $\R^2$. To see this, take $x,y$ distance 1 apart and let $v$ be a valuation with $v(p_1)=\{x\}$ and $v(p_2)=\{y\}$. Suppose \ts{lhs} is true at $x$. To have \ts{rhs} true at $x$ we need $x$ distance 1 from some $w$ in $\ol{v}(\Di p_1\wedge \Di p_2\wedge \Di(\Di p_1\wedge \Di p_2))$. Note that $\ol{v}(\Di p_1\wedge \Di p_2) = \{z_1,z_2\}$, the intersection points of the two unit circles centered at $x,y$. So we need $w$ distance 1 from $x$ with $w\in\{z_1,z_2\}$ 
and $w\in\Di(\{z_1,z_2\})$. But there is no such $w$. We now show that this formula is valid in $\R^m$ for $m\geq 3$. Suppose $v$ is some valuation on $\R^m$. If under this valuation \ts{lhs} is true at $x$, then there is $y\in\ol{v}(\Di p_2)$ with $d(x,y)=1$. Now $\ol{v}(\Di p_1\wedge\Di p_2) \supseteq C$ where $C$ is the intersection of the unit balls centered at $x,y$. To show that \ts{rhs} is true at $x$ it is enough to show that $x$ is distance 1 from some $w$ belonging to $C\cap\Di C$. Since $m\geq 3$ we have $C\cap\Di C = C$ since every point of $C$ is of distance 1 from another point of $C$.    
\end{proof}

For our next result, consider $\R^n$ as a graph with $R$ as its edge relation, so with an edge between $x$ and $y$ iff $d(x,y)=1$. There is an extensive literature on these graphs, and in particular it is known that their chromatic numbers $\chi(\R^n)$ are finite. Let $S^n$ be the unit sphere of $\R^n$ considered as a subgraph. Then the chromatic number $\chi(S^n)$ is finite for each $n$ and it is known that $\chi(S^n)\to \infty$, see \cite{chromatic-spheres}. Now for given $k\in\N$, consider the following formula 

\[\phi_k\,\,=\,\,\Box \bigvee_{i=1}^k (p_i\wedge \bigwedge_{i\neq j} \neg p_j) \,\,\imp\,\, \Di \bigvee_{i=1}^k (p_i\wedge \Di p_i).\]
\vspace{.5ex}

\begin{proposition}
\label{prp: chromatic}
$\R^n\models\phi_k$ iff $k<\chi(S^n)$.    
\end{proposition}

\begin{proof}
Write $\phi_k=\Box\psi_k\to\Di \mu_k$. First, suppose $\chi(S^n)\leq k$, so that there is a $k$-coloring of $S^n$. Then $S^n$ can be partitioned into sets $P_1,\ldots,P_k$, possibly with some empty, such that no element in a set $P_i$ is adjacent to another element in $P_i$. Let $v$ be a valuation for $\R^n$ so that $v(p_i)=P_i$ for each $i\leq k$. Then $\ol{v}(\psi_k)$ contains $S^n$. So everything distance~1 from the origin $o$ is contained in $\ol{v}(\psi_k)$, giving $o\in\ol{v}(\Box\psi_k)$. However, the $P_i$ give a coloring, so $P_i\cap\Di P_i=\emptyset$ for each $i\leq k$, hence $\ol{v}(\Di\mu_k)=\emptyset$. So $\phi_k$ is not valid in $\R^n$. 

Suppose $k<\chi(S^n)$. Let $v$ be a valuation and set $P_i=v(p_i)$ for $i\leq k$. If $x\in\ol{v}(\Box\psi_k)$, then the sphere of radius 1 centered at $x$ is contained in $\ol{v}(\psi_k)$. So the restriction of the $P_i$'s to this sphere partition it. But there is no $k$-coloring of this sphere since the assumption $k<\chi(S^n)$ says that there is none of the unit sphere centered at the 
origin. So there are two adjacent vertices $y,z$ belonging to some $P_i$ and lying on the unit sphere centered at $x$. This implies that $y\in P_i\cap\Di P_i\neq\emptyset$ and hence that $x\in\ol{v}(\Di \mu_k)$. It follows that $\phi_k$ is valid in $\R^n$. 
\end{proof}

Since $S^n$ can be embedded into $S^{n+1}$ we have that the chromatic numbers $\chi(S^n)$ are increasing. As we noted, it is known that this sequence converges to $\infty$. Thus there are infinitely many values of $n$ where $\chi(S^n)<\chi(S^m)$ for all $n<m$. Thus, we have the following
\Sp

\begin{corollary}\label{cor:equidist-antichain}
For each $n$ there is $n'$ so that $\Loge(\R^m)\not\subseteq\Loge(\R^n)$ for all $n'<m$, hence the set of logics $\Loge(\R^n)$ for $n\in\N$ contains an infinite anti-chain. 
\end{corollary}
\Sp

\begin{remark}
The technique used in Proposition~\ref{prp: chromatic} could be applied directly to the better known setting of the full unit distance graphs $\R^n$. By the De Bruijn–Erdős theorem, the chromatic number of $\R^n$ is equal to that of a bounded subgraph and $\Di^{r}(x)$ produces a ball of radius $r$ around $x$ for each $r>1$. Then modifying the formula $\phi_k$ used in Proposition~\ref{prp: chromatic} by replacing $\Di$ and $\Box$ with appropriate $\Di^r$ and $\Box^r$ yields a corresponding result. If it were known that the chromatic numbers of the spheres $S^n$, or of the full distance graphs $R^n$, were strictly increasing, we would have that the collection of all logics $\Loge(\R^n)$ for $n\in\N$ is an anti-chain.   
\end{remark}

\section{Comparing the logics for $\R$ and $\Q$}

In this section we compare the logics for $\R$ and $\Q$ with the $\Dg,\Dl$ and $\De$ modalities. The case for the $\De$ modality is the simplest.
\Sp

\begin{prp}\label{prop:dim1:r=1}
$\Loge(\R)=\Loge(\Q)$.
\end{prp}
\begin{proof}
For any $x\in\R$, the subframe generated by $x$ is isomorphic to $\Z$.
\end{proof}

\subsection{The $\Dg$ modality}
This subsection deals solely with the $\Dg$ modality which we write $\Di$, and we write the relation $R_{>1}$ as $R$. We show that $\Logg(\Q)$ is strictly contained in $\Logg(\R)$. That there is a formula valid in $\R$ and not in $\Q$ is given in the following. 
\Sp

\begin{prp}\label{prp:less-RisnotincludedinQ}
$\Logg(\R)$ is not contained in $\Logg(\Q)$. \end{prp}

\begin{proof}
In $\R$, if two maximal anti-cliques, i.e. closed unit intervals, cover a third that is distinct from either one, then they must intersect. This is not the case in $\Q$ since $(\pi-1,\pi)\cup(\pi,\pi+1)$ are disjoint but cover the interval [3,4]. The formula 
\[\bigwedge_{i=1}^3\Box^2 (p_i\leftrightarrow \neg \Di p_i) \wedge \bigwedge_{i\neq j} \Di^2 (p_i \wedge \neg p_j)\wedge \Box^2 (p_1\rightarrow p_2\vee p_3)
\,\,\to\,\, \Di^2 (p_2\wedge p_3)
\]
says that if the valuations of $p_1,p_2,p_3$ are maximal anti-cliques that are pairwise distinct and the valuation of $p_1$ is contained in the union of those of $p_2$ and $p_3$, then the valuations of $p_2$ and $p_3$ intersect non-trivially. So this formula is valid in $\R$ but not in $\Q$.
\end{proof}

Showing that the logic of $\Q$ is contained in that of $\R$ is more involved. For this, we must show that any formula valid in $\Q$ is valid in $\R$, or, using the contrapositive, that any formula $\phi$ that is satisfiable in $\R$ under some valuation is satisfiable in $\Q$. 
The informal idea is that given a valuation in $\R$ and a formula $\phi$, we associate a finite number of ``pivot points'', and then slightly perturbing these points we obtain a valuation in $\Q$ preserving satisfiability of $\phi$.
Throughout the remainder of this subsection, we assume that $\phi$ is a fixed formula, and that $\Psi$ is the set of sub-formulas of $\phi$. 
\Sp

\begin{definition}
For a valuation $v$ on $\R$, set $\ol{v}(\Psi)=\{\ol{v}(\psi)\mid\psi\in\Psi\}$. Let $B_v$ be the Boolean subalgebra of the powerset of $\R$ generated by $\ol{v}(\Psi)$ and let $X_v$ be the set of all real numbers that arise as infima or suprema of members of $B_v$. We say the valuation $v$ is \emph{special} if $X_v\subseteq\Q$.   
\end{definition}
\Sp

\begin{lemma}\label{lem:RvsQ:>1:rationalVal}
If $\phi$ is satisfiable in $\R$ under some valuation $v$, then it is satisfiable under a special valuation.    
\end{lemma}

\begin{proof}
Suppose $\alpha$ is an order-automorphism of $\R$ so that it and its inverse preserve the property of being distance $>1$ apart, and with $\alpha(x)$ rational for each $x\in X_v$. Consider the valuation $u$ given by $u(p)=\alpha[v(p)]$. 
For $\psi\in\Psi$ we have $\ol{u}(\psi)$ is the image $\alpha[\ol{v}(\psi)]$. Since $\alpha$ is an order-automorphism, it preserves existing infima and suprema, and it follows that $u$ is a special valuation. That $\phi$ is satisfiable for $v$ means that $\ol{v}(\phi)$ is non-empty, so $\ol{u}(\phi)=\alpha[\ol{v}(\phi)]$ is non-empty, giving that $\phi$ is satisfiable for $u$. 

To construct such an $\alpha$, for any real number $x$, let $\lfloor x\rfloor$ be the largest integer beneath $x$ and the $r(x)$ be the remainder $x-\lfloor x\rfloor$. Let $R$ be the set of remainders of elements of $X_v$. Since $\Psi$ is finite, so is $B_v$, and thus $R$ is finite. So there is an increasing bijection $\beta:[0,1)\to[0,1)$ with $\beta(x)$ rational for each $x\in R$. Set $\alpha(x)=\lfloor x\rfloor+\beta(r(x))$. Then $\alpha$ is an order-automorphism and $\alpha(x)$ is rational for each $x\in X_v$. It remains to see that $\alpha$ and its inverse preserve the property of being distance $>1$ apart. This follows since for $x\leq y$ we have $|y-x|\leq 1$ iff $\lfloor x\rfloor = \lfloor y\rfloor$ or $\lfloor y\rfloor=\lfloor x\rfloor+1$ and $r(y)\leq r(x)$.
\end{proof}
 
\begin{remark}
\label{rem: auto <1}
We will use this same construction in a later result. Specifically, for a finite set $X$ of real numbers, we will require an order-automorphism $\alpha$ of $\R$ with $\alpha(x)$ rational for each $x\in X$ and such that $\alpha$ and its inverse preserve the property of being distance $<1$ apart. The reader may wish to verify that the construction described above also has this property for distance $<1$.
\end{remark} 
\Sp

From this point onward, we assume that $\phi$ is satisfiable in $\R$ under the special valuation $v$. We denote $X_v$ simply by $X$ and note that $v$ being special means that each member of $X$ is rational. Let $\mathcal{I}$ be the collection of open intervals that have members of $X$ or $\pm\infty$ as their endpoints and do not contain any members of $X$. Note that the members of $\mathcal{I}$ together with $X$ form a partition of $\R$.
\Sp

\begin{lemma} 
\label{lem: I subset}
If $I\in\mathcal{I}$ and $\Di\psi\in\Psi$, either $I$ is disjoint from $\ol{v}(\Di\psi)$ or $I\subseteq\ol{v}(\Di\psi)$. 
\end{lemma}

\begin{proof}
For a formula $\phi$, $\ol{v}(\Box\phi)=\{x\mid d(x,y)\leq 1\mbox{ for each }y\not\in \ol{v}(\phi)\}$ is the intersection of closed intervals of length 2, so is either a closed interval $[u,v]$ or all of $\R$ when the intersection is of the empty family. Since $\ol{v}(\neg\Di\psi)=\ol{v}(\Box\neg\psi)$, we have $\ol{v}(\neg\Di\psi)$ is either $\R$ or is a closed interval $[u,v]$. If $\ol{v}(\neg\Di\psi)=\R$, then $\ol{v}(\Di\psi)$ is empty, hence is disjoint from $I$. Suppose then that $\ol{v}(\neg\Di\psi)$ is the closed interval $[u,v]$. Clearly $u,v$ are the infimum and supremum of the set $\ol{v}(\neg\Di\psi)$, so they belong to $X$. 
Since $I$ is an open interval that does not contain any members of $X$, it is either contained in $[u,v]$ or disjoint from $[u,v]$, hence is either disjoint from $\ol{v}(\Di\psi)$ or contained in it. 
\end{proof}


\begin{prp}\label{prp:less-Qis-includedinQ}
$\Logg(\Q)\subseteq\Logg(\R)$.    
\end{prp}
\setcounter{claim}{0}
\begin{proof}
Assume that $\phi$ is satisfiable in $\R$ under the special valuation $v$. Let $\Psi$, $B$, $X$ and $\mathcal{I}$ be as described above. We will construct a valuation $w$ on $\Q$ and show that $\phi$ is satisfiable in $\Q$ under this valuation.

Suppose that $\phi$ has $k$ variables. For each $a\in \mR$, let $\tau(a)=\{i\mid a\in v(p_i)\}$, there there are $2^k$ possible sets of the form $\tau(a)$.  For each $I\in\mathcal{I}$, partition $I\cap\Q$ into dense pieces $V_{I,\tau}$ where $\tau$ ranges over the sets $\tau(a)$ for $a\in I$ and let $Z$ be the relation from $\R$ to $\Q$ where $a\,Z\,b$ iff either $a\in X$ and $a=b$ or $a\in I$ and $b\in V_{I,\tau(a)}$. We then define a valuation $w$ on $\Q$ by setting $w(p)=Z[v(p)]$.
\vspace{1ex}

\begin{claim}\label{cl:th1:1}
If $a\,Z\,b$, then $a\in v(p)$ iff $b\in w(p)$.    
\end{claim}

\begin{proof}[Proof of claim]
If $a\in v(p)$, then since $a\,Z\,b$ we have $b\in w(p)$. Conversely, suppose $b\in w(p)$. If $a\in X$, then since $a\,Z\,b$ we have $a=b$. Since $b\in w(p)$ we have $c\,Z\,b$ for some $c\in v(p)$. But $b=a$, so we must have $c=a$, hence $a\in v(p)$. Suppose $a\in I$ for some $I\in\mathcal{I}$. Since $a\,Z\,b$ and $a\not\in X$ we have $b\in V_{I,\tau(a)}$. As we have $b\in w(p)$, there is $c\in v(p)$ with $c\,Z\,b$. Since $b\in I$, we must have $b\in V_{I,\tau(c)}$. Since the $V_{I,\tau}$ are a partition, we have $\tau(c)=\tau(a)$. But $\tau(a)=\{p\mid a\in v(p)\}$, and since $c\in v(p)$, it follows that $a\in v(p)$.
\end{proof}

\begin{claim}\label{cl:th1:2}
For $a\,Z\,b$ and $\psi\in\Psi$, we have $\R,a\models_v\psi$ iff $\Q,b\models_w\psi$.
\end{claim} 

\begin{proof}[Proof of claim] The case when $\psi$ is a variable is given by 
Claim \ref{cl:th1:1}.
The cases when the outermost connective of $\psi$ is boolean are trivial. It remains to consider $\psi=\Di\gamma$. Since $a\,Z\,b$ implies that $a=b$ or that $a,b$ belong to the same $I\in\mathcal{I}$, by 
Lemma~\ref{lem: I subset} we have

\[
\R,a\models_v\Di\gamma\,\,\Leftrightarrow\,\,\R,b\models_v\Di\gamma
\]
\vspace{-1ex}

``$\Rightarrow$'' Assume $\R,a\models_v\Di\gamma$, so by the above, $\R,b\models_v\Di\gamma$. Then there is $c\in\R$ with $d(b,c)>1$ and  $\R,c\models_v\gamma$. If $c\in X$, then $c\,Z\,c$, so $\Q,c\models_w\gamma$ by the inductive hypothesis, hence $\Q,b\models_w\Di\gamma$. If $c\not\in X$, then $c\in I$ for some $I\in\mathcal{I}$. Since $V_{I,\tau(c)}$ is dense in $I$, and $d(b,c)>1$, there is $e\in V_{I,\tau(c)}$ with $d(b,e)>1$. Then $c\,Z\,e$ and the inductive hypothesis gives $\Q,e\models_w\gamma$. Since $d(b,e)>1$ we have $\Q,b\models_w\Di\gamma$.

``$\Leftarrow$'' Suppose $\Q,b\models_w\Di\gamma$. From above, it is enough to show that $\R,b\models_v\Di\gamma$. Our assumption gives $e\in \Q$ with $d(b,e)>1$ and $\Q,e\models_w\gamma$.
If $e\in X$, then $e\,Z\,e$. Then the inductive hypothesis gives $\R,e\models_v\gamma$, and hence $\R,b\models_v\Di\gamma$. 
If $e\not\in X$, then $e\in I\cap\Q$ for some $I\in\mathcal{I}$. Since the sets $V_{I,\tau}$ partition $I\cap\Q$ there is $c\in I$ with $e\in V_{I,\tau(c)}$, hence with $c\,Z\,e$. Then, by the inductive hypothesis, $\R,c\models_v\gamma$, so $c\in\ol{v}(\gamma)$. If $\ol{v}(\gamma)$ is unbounded, then trivially $\R,b\models_v\Di\gamma$. Otherwise let $\ell$ and $u$ be the infimum and supremum of $\ol{v}(\gamma)$ and note that $\ell,u$ belong to $X$. Since $c\in\ol{v}(\gamma)$ and $c\in I$ we have $\ell<c<u$, and since $I$ is an interval that does not contain $\ell,u$ we have $I\subseteq (\ell,u)$. In particular, since $e\in I$ we have $\ell<e<u$. Since $d(b,e)>1$, either $d(b,\ell)>1$ or $d(b,u)>1$. Since there are elements of $\ol{v}(\gamma)$ arbitrarily close to $\ell$ and to $u$, there is $f\in\ol{v}(\gamma)$ with $d(b,f)>1$. Thus $\R,b\models_v\Di\gamma$.
\end{proof}

Returning to the proof of the proposition, we assumed that $\phi$ was satisfiable in $\R$ under the valuation $v$. So there is $a\in\R$ with $\R,a\models_v\phi$. There is $b\in\Q$ with $a\,Z\,b$, indeed, if $a\in X$ take $b=a$, and if $a\in I$ take any $b\in V_{I,\tau(a)}$. Then by 
Claim \ref{cl:th1:2}
we have $\Q,b\models_w\phi$. So $\phi$ is satisfiable in $\Q$ under the valuation $w$. 
\end{proof}
 
Combining Propositions~\ref{prp:less-RisnotincludedinQ} and~\ref{prp:less-Qis-includedinQ}, we have 
\Sp

\begin{theorem} \label{thm:combined}
The logic $\Logg(\Q)$ is strictly contained in $\Logg(\R)$. 
\end{theorem}

\subsection{The $\Dl$ modality}

This subsection deals solely with the $\Dl$ modality which we write $\Di$. We show that in this modality, the logic of $\Q$ is strictly contained in that of $\R$. The first step in the proof uses a formula that expresses a form of connectivity. Let $\phi$ be the formula

\begin{equation*}
   \Box(\Di \Box p \vee \Di \Box \neg p) \,\,\imp\,\, \Box p \vee \Box \neg p 
\end{equation*}
\vspace{-1ex}

\begin{prp}
$\Logl(\R)\not\subseteq\Logl(\Q)$.     
\end{prp}

\begin{proof}
We show that $\phi$ is valid in $\R$ but not in $\Q$. To show $\phi$ is not valid in $\Q$ consider a valuation $v$ for $\Q$ with $v(p)=(x,\infty)$ for some irrational $x$. Then  

\[
\begin{array}{ccccc}
\ol{v}(p)=(x,\infty)&&\ol{v}(\Box p)=(x+1,\infty)&&\ol{v}(\Di\Box p)=(x,\infty)\vspace{1ex}\\
\ol{v}(\neg p)=(-\infty,x)&&\ol{v}(\Box \neg p)=(-\infty,x-1)&&\ol{v}(\Di\Box \neg p)=(-\infty,x)
\end{array}
\]
\Sp

\noindent 
Then $\ol{v}$ of the premise of $\phi$ is all of $\Q$, but $\ol{v}$ of the conclusion of $\phi$ is missing all elements within distance $1$ of $x$. 
Thus $\phi$ is not valid in $\Q$. We next show that $\phi$ is valid in $\R$. Assume that $x$
\[
\mR,x\mo_\val\Box(\Di\Box p \vee \Di\Box \neg p).    
\]
Observe that the value of any formula of form $\Di\psi$ is an open set: indeed, $\vext(\Di\psi)$ is the union of open intervals. Let $I$ be the interval $(x-1,x+1)$. Then $I$ is contained in the union of two open sets $\vext(\Di\Box p)$ and $\vext(\Di\Box\neg p)$. These sets are disjoint, since for any formula $\psi$,   $\vext(\Di\Box \psi)\subseteq \vext(\psi)$. Since $I$ is connected, $I\cap \vext(\Di\Box p)$ is empty or $I\cap \vext(\Di\Box \neg p)$ is empty. In the first case, $I\subseteq \vext(\Di\Box \neg p)\subseteq \val(\neg p)$, and we have $\mR,x\mo_\val\Box\neg p$; the latter case implies $\mR,x\mo_\val\Box p$. Hence $\phi$ is valid in $\mR$.
\end{proof}

The formula $\phi$ has a more general interpretation that is perhaps of interest, and we discuss this in the following remark. 
\Sp

\begin{remark}
By a \emph{graph} $G=(X,R)$ we mean here a set with a reflexive, symmetric relation. We say a set of the form $R(x)$ is \emph{basic open} and that a set is \emph{open} if it is the union of basic open sets. Continuing the topological metaphor, we say that a subset $S$ of $X$ is \emph{connected} (in this new sense) if $S$ being contained in the union of two disjoint open sets implies that it is contained in one of the sets. It is not difficult to show that in a graph $G$, a set $S$ is open iff $S=\Di\Box S$ and that each set of the form $\Di\Box S$ for some $S$ is open. It follows that $G\models\phi$ iff every basic open set is connected. The proposition above reflects that basic open sets in $\R$ are connected, but not so in $\Q$.
\end{remark}






\subsubsection{$\Logl(\mQ)$ is included in $\Logl(\mR)$}

We mainly following the same strategy as for the farness logic, but the proof requires more methods from modal and classical logic. To state our first step, we introduce the following standard terminology.
For a formula $\vf$, let $\md(\vf)$ denote its {\em modal depth}, which is recursively defined with falsity $\bot$ and variables having modal depth 0, $\psi\to\chi$ having the maximum of the modal depth of $\psi,\chi$ and $\Di\psi$ having modal depth one greater than that of $\psi$. The following fact is well known. 

\Sp
\begin{proposition}\label{prop:modal-depth}
A formula $\vf$ is satisfiable at $x$ in $(X,R)$ iff it is satisfiable at $x$ in the subframe $Y=\bigcup\{R^n(x)\mid n\leq \md{(\phi)}\}$. 
\end{proposition} 
\Sp
\setcounter{claim}{0}

Our aim is to show that 
$\Logl(\mQ) \subseteq     \Logl(\mR)$ by showing that any formula $\vf$ that is satisfiable in $\R$ at some point $x$ is satisfiable in $\Q$. We assume without loss of generality that $x=0$ and let $\Psi$ be the set of subformulas of $\phi$.

\Sp

\begin{claim}\label{cl:th2:1}
There is an interval $V=(-d,d)$ and a valuation $v$ on $V$ with $V,0\models_v\phi$.
\end{claim}

\begin{proof}[Proof of claim]
If the modal depth of $\vf$ is 0, then any $d>0$ will suffice. Otherwise, by Proposition~\ref{prop:modal-depth} we can choose $d$ to be the modal depth of $\phi$. 
\end{proof}


\begin{claim}\label{cl:th2:2}
$\ol{v}(\Di\psi)$ is the disjoint union of finitely many open intervals $\clI_\psi$. 
\end{claim}

\begin{proof}[Proof of claim]
If $x\in\ol{v}(\Di\psi)$ then there is $y\in\ol{v}(\psi)$ with $x\in(y-1,y+1)$ and with the intersection of this interval with $V$ contained in $\ol{v}(\Di\psi)$. In particular, $\ol{v}(\Di\psi)$ is open in $\R$ and so can be uniquely expressed as the union of a family $\clI_\psi$ of at most countably many pairwise disjoint non-empty open intervals. At most two of these intervals have points outside of $\ol{v}(\psi)$ and it follows that $\clI_\psi$ is finite.    
\end{proof}



Recall our usage of ``close'' to mean two points distance strictly less than 1 from each other. 
For a formula $\psi$ we define the set $D_\psi$ of $\psi$-sandwiched points by declaring $x\in D_\psi$ if there are points in $\ol{v}(\psi)$ strictly smaller and larger than $x$ but close to $x$. 
\Sp

\begin{claim}\label{cl:th2:3}
$D_\psi$ is the disjoint union of finitely many open intervals $\clJ_\psi$. 
\end{claim}

\begin{proof}[Proof of claim]
Notice that $D_\psi$ is open, so it is the union of a unique family $\clJ_\psi$ of disjoint non-empty open intervals. We claim that if $(a,b)$, $(c,d)\in\clJ_\psi$ and $b\leq c$, then $d(a,d)> 1$. Assume not, so  $d(a,d)\leq 1$, and in particular $d(a,b)<1$. We have $b\notin D_\psi$. Consider two cases. First, assume that there are no points in $\vext(\psi)$ close to $b$ from the left, so $(b-1,b)\cap \vext(\psi)=\emp$.  
Choose $x\in(a,b)$. Since $x\in D_\psi$ there is $y\in\ol{v}(\psi)$ with $x-1<y<x$. Then $a-1<x-1<y\leq b-1<a$. So $y$ is in $\ol{v}(\psi)$ and is close to $a$ from the left. Since $a$ is not $\psi$-deep, there are no points in $\vext(\psi)$ close to $a$ from the right. So  $(a,a+1)\cap \ol{v}(\psi)$ is empty, and consequently $(b-1,a+1)\cap \ol{v}(\psi)$ is empty. Consider a point $z$ in $(c,d)$. We have $b-1<z-1$ and since $d(a,d)\leq 1$ also $z<a+1$. So $z$ in not $\psi$-deep, which is a contradiction. 

Now assume that there are no points in $\vext(\psi)$ close to $b$ from the right. Let $a<x<b$. Since $x$ is $\psi$-deep, there is $y\in \vext(\psi)$ with $a<x<y\leq b$. We have $d-y<1$, and since $d$ in not $\psi$-deep, there are no points in $\vext(\psi)$ close to $d$ from the right. Choose $z\in (c,d)$. Since $z\in D_\psi$, there is $w\in\ol{v}(\psi)$ that is close to $z$ from the right. This $w$ must be in $(z,d]$, so is close to $c$ from the right since $d(a,d)\leq 1$. Thus there are elements of $\ol{v}(\psi)$ on either side of $c$, giving $c\in D_\psi$, a contradiction. 
\end{proof} 

For each formula $\psi$ we have that $\ol{v}(\Di\psi)$  and $D_\psi$ are the the union of finitely many pairwise disjoint intervals $\clI_\psi$ and $\clJ_\psi$ respectively. Let $X$ be the set of all points that occur as an endpoint of some interval occurring in $\clI_\psi$ or $\clJ_\psi$ for some $\psi$ with $\Di\psi$ among the subformulas $\Psi$  of $\phi$. Clearly, $X$ is finite and in view of Lemma \ref{lem:RvsQ:>1:rationalVal}  and Remark~\ref{rem: auto <1} that follows it,  
we can assume that each member of $X$ is rational. So there is a integer $N\geq 1$ so that each member of $X$ is an integer multiple of $r=\frac{1}{N}$. Let $\clK$ be the set of open intervals in $V$ of form $(mr,mr+r)$ with $m$ integer, and let $P$ be the set of their endpoints in $V$.
\Sp

\begin{claim}\label{cl:th2:4}
If $I\in\clK$ and $J$ belongs to $\clI_\psi$ or $\clJ_\psi$ for some $\psi$ with $\Di\psi\in\Psi$,  either $I$ is either contained in $J$ or disjoint from $J$. 
\end{claim}

\begin{proof}[Proof of claim]
Otherwise an endpoint of $J$ would be an internal point of $I$. 
\end{proof}

\begin{claim}\label{cl:th2:5}
If $x,y$ belong to an interval in $\clK$ and $\Di\psi\in\Psi$ then 
\begin{equation*}\label{eq:RvsQ:<1:stepV1refined}
V,x\mo_{\val} \Di  \psi\,\,  \Leftrightarrow\,\, V,y\mo_{\val}  \Di  \psi  
\end{equation*}
\vspace{-4ex}
\begin{equation*}\label{eq:RvsQ:<1:deep}
x\in D_\psi\,\,\Leftrightarrow\,\, y\in D_\psi
\end{equation*}
\vspace{-2ex}
\end{claim}

\begin{proof}
For the first statement we have  $V,x\models_v\Di\psi$ iff $x\in\ol{v}(\Di\psi)$, which occurs iff $x$ belongs to an interval in $\clI_\psi$. The result follows since $x,y$ either both belong to an interval in $\clI_\psi$ or neither does. The second statement is similar, either both belong to an interval in $\clJ_\psi$ or neither does. 
\end{proof}

Due to the previous two claims, intervals in $\clK$ will play an important role in later arguments. We next collect some simple arithmetic consequences of the fact that they divide $V$ into pices of equal length $r$. Suppose $I=(a,a+r)\in\clK$. 
\vspace{4ex}

\begin{center}
\begin{tikzpicture}
\draw (-6,0)--(6,0); 
\node at (-.6,0) {$($};
\node at (.6,0) {$)$};
\node at (-.6,-.5) {$_a$};
\node at (.6,-.5) {$_{a+r}$};
\node at (-5.8,0) {$($};
\node at (-4.6,0) {$)$};
\node at (-5.8,-.5) {$_{a-1}$};
\node at (-4.6,-.5) {$_{a-1+r}$};
\node at (-4.5,0) {$($};
\node at (-3.4,0) {$)$};
\node at (-3.2,-.5) {$_{a-1+2r}$};
\node at (5.8,0) {$)$};
\node at (4.6,0) {$($};
\node at (5.8,-.5) {$_{a+1+r}$};
\node at (4.6,-.5) {$_{a+1}$};
\node at (4.5,0) {$)$};
\node at (3.4,0) {$($};
\node at (3.3,-.5) {$_{a+1-r}$};
\end{tikzpicture}   
\end{center}
\vspace{3ex}

\noindent We say that $I$ is \emph{close} to an interval $J\in\clK$ if there are $x\in I$ and $y\in J$ with $x,y$ close. If they do not lie outside the range $(-d,d)$, the intervals that are close to $I$ are shown in the diagram above. An interval that is close to $I$ and is not $(a-1,a-1+r)$ or $(a+1,a+1+r)$ is \emph{very close} to $I$. The relations of close and very close are symmetric and we say intervals are close or very close. We say that $x,y$ are \emph{very close} if there are $I,J\in\clK$ that are very close with $x\in I$ and $y\in J$. The following is evident. 
\Sp


\begin{claim}\label{cl:th2:6}
Let $I,J\in\clK$. 
\vspace{1ex}
\begin{enumerate}
\item $I,J$ are close iff every point in one interval is close to a point in the other.
\item $I,J$ are very close iff every point in one is close to every point in the other. 
\end{enumerate}
\end{claim}
\Sp

Having established the setup, we make a first step to construct a valuation on $\Q$ that satisfies $\phi$. 
\Sp

\begin{claim}\label{cl:th2:7} 
There is a countable subframe $U$ of $V$ with valuation $u(p)=v(p)\cap U$ so that the following hold
\end{claim}
\Sp

\begin{enumerate}
\item $U$ is a dense linear order without endpoints
\item $U$ contains all endpoints of intervals in $\clK$
\item for all $x\in U$ and formulas $\psi$ we have $U,x\models_u\psi$ iff $V,x\models_v\psi$
\item for each formula $\psi$, 
$\ol{u}(\psi)$ is dense in $\ol{v}(\psi)$
\end{enumerate}
\Sp

\begin{proof}[Proof of claim]
We make use of the downward L\"{o}wenheim-Skolem theorem. First, it is well-known that each subset of the reals has an at most countable dense subset and that the union of at most countably many countable sets is countable. So we can find a countable subset $Z$ of $V$ that contains endpoints of intervals in $\clK$ and contains a dense subset of $\ol{v}(\psi)$ for each formula $\psi$. We consider a countable language in which $V$ is a model where closeness $R$ and the partial ordering of $V$ are binary predicates, there are constant symbols $c_z$ interpreted as $z$ for each $z\in Z$, and each interval $I\in\clK$ is viewed as a unary predicate of $V$. For each variable $p$ let $Q_p$ be a unary predicate symbol that is interpreted as $v(p)$. Then the downward L\"{o}wenheim-Skolem theorem gives a countable elementary substructure $U$ of $V$. We use $*$ to indicate the interpretations of the components of the language in the substructure $U$. So 
\begin{align*}
R^*\,&=\mbox{ the restriction of $R$ to $U$}\\
\leq^*\,&=\mbox{ the restriction of $\leq$ to $U$}\\
c_z^*\,&=\,z \mbox{ for each }z\in Z\\
I^*\,&=\,I\cap U\mbox{ for each }I\in\clK\\
Q_p^*\,&=\,Q_p\cap U\mbox{ for each variable }p
\end{align*}

\noindent Then, ignoring parts of the structure, $(U,R^*)$ is a frame and we let $u$ be the valuation on this frame with $u(p)=v(p)\cap U$ for each variable $p$.

The first part of our claim follows from the fact that $V$ is a dense linear order without endpoints, and this is a first-order property. The second statement follows since $U$ contains $Z$ and $Z$ contains the endpoints of intervals in $\clK$. The third statement follows from that fact that modal satisfaction is a first-order property. In more detail, for a modal formula $\psi$, point $x\in V$, and a valuation $v$ in $V$, there is a first-order formula $\hat{\psi}$ in the language containing a binary predicate for the relation of the frame and unary predicates for the sets $v(p)$ where $p$ is a variable in $\psi$ so that 
\[
V,x\models_v\psi\quad\Leftrightarrow\quad V\models\hat{\psi}(x)
\]

\noindent Here the first satisfaction symbol is in the modal sense and the second is in the first-order sense. See e.g., \cite[Proposition 2.47]{BDV} for a proof of this statement. Then, if $x\in U$, since the first-order structures $U$ and $V$ are elementarily equivalent where the unary predicates for $v(p)$ are interpreted as $u(p)$, we have $V\models\hat{\psi}(x)$ iff $U\models\hat{\psi}(x)$ and repeating the previous argument, $U\models\hat{\psi}(x)$ iff $U,x\models_u\psi$.

For the final statement, for $x\in U$ the third statement gives $x\in\ol{u}(\psi)$ iff $x\in\ol{v}(\psi)$, and it follows that $\ol{u}(\psi)=\ol{v}(\psi)\cap U$. That $\ol{u}(\psi)$ is dense in $\ol{v}(\psi)$ then follows since $Z$ is contained in $U$ and $Z$ contains a dense subset of $\ol{v}(\psi)$ for each formula $\psi$.
\end{proof} 

Since by part (1) of  
Claim \ref{cl:th2:7}
$U$ is a dense linear order without endpoints, so is each open interval $(a,b)$ in $U$. So by Cantor's theorem, for each interval $I=(a,b)$ in $\clK$ there is an order-isomorphism $g_I:I\cap U\to I\cap\Q$. By part (2) of 
Claim \ref{cl:th2:7}
$U$ contains the endpoints of intervals in $\clK$. So taking the union of these maps and setting $W=V\cap\Q$ we obtain an order-isomorphism $g:U\to W$ that restricts to the identity on the endpoints of intervals in $\clK$. While $g$ need not preserve distances, we have the following. 
\Sp

\begin{claim}\label{cl:th2:8} 
Let $x,y\in U$, $I\in\clK$, $a$ be an endpoint of $I$ and $\Di\psi\in\Psi$. Then  
\Sp

\begin{enumerate}
\item $x\in I$ iff $g(x)\in I$
\item if one of $x,y$ is an endpoint of $I$ then $x,y$ are close iff $g(x),g(y)$ are close
\item if $x,y$ are very close, then $g(x),g(y)$ are close. 
\item if $a$ is a limit point of $Y\subseteq I\cap U$, then $a$ is a limit point of $g[Y]$. 
\end{enumerate}
\end{claim}


\begin{proof}[Proof of claim]
The first statement follows from the definition of $g$. For the second statement, if $x$ is an endpoint of $I$, then $x,y$ are close iff $y$ belongs to one of the family of intervals shown in the illustration before 
Claim \ref{cl:th2:6}.
Since $g(x)=x$, we have $g(x),g(y)$ are close iff $g(y)$ belongs to one of these same intervals. The result then follows from the first statement. For the third statement, if $x\in I=(a,a+r)$, then $x,y$ are close iff $y$ belongs to one of the intervals not at the end in the illustration before 
Claim \ref{cl:th2:6}.
Since $g(x)\in I$, we have $g(x),g(y)$ close iff $g(y)$ belongs to one of these same intervals, and the result follows by the first statement. For the final statement, $a$ being a limit point of $Y$ is equivalent to $a$ being the meet or join of $Y$, and $g_I$ is an order-isomorphism so it preserves meets and joins. 
\end{proof} 

Let $gu$ be the valuation on $W$ given by $gu(p)=g[u(p)]$ for each variable $p$. We will show that for all $x$ in $U$, and for all subformulas $\sigma$ of $\vf$, we have 
\begin{equation}\label{eq:RvsQ:<1:step4}
U,x\mo_\valu \sigma\quad\Leftrightarrow\quad W,g(x)\mo_{g\valu} \sigma,
\end{equation}

\noindent The proof is by induction on the complexity of $\psi$. The base cases that $\psi$ falsity or a variable follows trivially or from the definition of the valuation. The case of implication is trivial. It remains to consider the case when $\sigma$ is $\Di\psi$. 
\Sp

\begin{claim}\label{cl:th2:9} 
$W,g(x)\models_{gu}\Di\psi\, \Rightarrow\, U,x\models_u\Di\psi$. 
\end{claim}

\begin{proof}[Proof of claim]
Assume $W,g(x)\mo_{g\valu}\Di  \psi$. Since $g$ is a bijection, there is $y\in U$ with $g(y)$ close to $g(x)$ and with $W,g(y)\mo_{g\valu}\psi$. So by the inductive hypothesis, $U,y\models_u\psi$. If either $x,y$ is an endpoint of an interval in $\clK$, then by part (2) of 
Claim \ref{cl:th2:8}
we have that $x,y$ are close, so $U,x\models_u\Di\psi$ as required. If $x\in I$ and $y\in J$ for some $I,J\in\clK$, then by part (1) of 
Claim \ref{cl:th2:8}
we have $g(x)\in I$ and $g(y)\in J$. Since $g(x)$ and $g(y)$ are close in $\Q$, and hence also in $\R$, by 
Claim \ref{cl:th2:6}
the intervals $I,J$ are close, so 
Claim \ref{cl:th2:6}
gives that $y$ is close to some element $z\in I$. Then $U_z\models_u\Di\psi$, so by
Claim \ref{cl:th2:7}
we have $V,z\models_\Di\psi$. Since $x,z$ belong to the same interval in $\clK$, by 
Claim \ref{cl:th2:5}
we have $V,x\models_v\Di\psi$, hence by 
Claim \ref{cl:th2:7},
$U,x\models_u\Di\psi$ as required. 
\end{proof}




\begin{claim}\label{cl:th2:10}
$U,x\models_u\Di\psi\,\Rightarrow\,W,g(x)\models_{gu}\Di\psi$.
\end{claim}

\begin{proof}[Proof of claim] 
Since $U,x\models_u\Di\psi$ there is $y$ close to $x$ with $U,y\models_u\psi$, and for any such $y$ the inductive hypothesis gives $W,g(x)\models_{gu}\psi$.  Suppose first that either $x$ is an endpoint of an interval in $\clK$ or that $y$ can be choosen to be an endpoint of an interval in $\clK$ or to be very close to $x$. Then since $x,y$ are close, it follows by parts (2) and (3) of 
Claim \ref{cl:th2:8}
that $g(x)$ and $g(y)$ are close, so $W,g(x)\models_{gu}\Di\psi$ as required. 

It remains to consider the case where $x$ belongs to some interval $(a,a+r)\in\clK$ and every $y$ in $\ol{u}(\psi)$ that is close to $x$ is in $I^-=(a-1,a-1+r)$ or $I^+=(a+1,a+1+r)$. Symbolically, we have 
\[
R_{<1}(x)\,\cap\,\ol{u}(\psi)\,\subseteq\, I^-\cup I^+.
\]
\vspace{-1ex}

\noindent Suppose $z$ is any real number belonging to $(a,a+r)$. With the possible exception of elements of $I^-\cup I^+$, an element of $V$ is close to $z$ iff it is close to $x$. Thus, 
\[
R_{<1}(z)\,\cap\,\ol{u}(\psi)\,\subseteq\, I^-\cup I^+.
\]
By part (4) of 
Claim \ref{cl:th2:7},
$\ol{u}(\psi)$ is dense in $\ol{v}(\psi)$ and it follows that 
\[
R_{<1}(z)\,\cap\,\ol{v}(\psi)\,\subseteq\, I^-\cup I^+.
\]

Consider the following sets 
\begin{align*}
S^- \,&=\,\{z\in (a,a+r) \mid z\mbox{ is close to some $y$ in $\ol{v}(\psi)\cap I^-$}\} \\
S^+ \,&=\,\{z\in (a,a+r) \mid z\mbox{ is close to some $y$ in $\ol{v}(\psi)\cap I^+$}\} 
\end{align*}

\noindent Note that each of these sets is open. Also, we assumed $U,x\models_u\Di\psi$, so by part (3) of 
Claim \ref{cl:th2:7}
we have $V,x\models_v\Di\psi$, hence $x\in\ol{v}(\Di\psi)$. It follows from 
Claim \ref{cl:th2:4}
that $(a,a+r)$ is contained in $\ol{v}(\Di\Psi)$. So each $z\in (a,a+r)$ has an element of $\ol{v}(\psi)$ close to it, and therefore $(a,a+r)$ is contained in $S^-\cup S^+$. 

Recall that a point of $V$ is $\psi$-sandwiched if there are points on both sides of the point that are close to it and in $\ol{v}(\psi)$. By 
Claim \ref{cl:th2:4},
if $I\in\clK$, then either every point in $I$ is $\psi$-sandwiched or none are. If $x$ is $\psi$-sandwiched, hence all points of $(a,a+r)$ are $\psi$-sandwiched, then each of $S^-$ and $S^+$ is equal to $(a,a+r)$. If $x$ is not $\psi$-sandwiched, hence no point of $(a,a+r)$ is $\psi$-sandwiched, then the sets $S^-$ and $S^+$ are disjoint, and since they are open and cover the connected set $(a,a+r)$, one of them is equal to $(a,a+r)$. We assume without loss of generality that $S^+=(a,a+r)$. 

Since $S^+=(a,a+r)$ there are points $z$ arbitrarily close to $a$ with $z$ close to some point $y$ in $\ol{v}(\psi)\cap (a+1,a+r+1)$. It follows that there are points $y\in\ol{v}(\psi)\cap I^+$ that are arbitrarily close to $a+1$. So $a+1$ is a limit point of $\ol{v}(\psi)\cap I^+$. By part (4) of 
Claim \ref{cl:th2:7},
it follows that $\ol{u}(\psi)\cap I^+$ is dense in $\ol{v}(\psi)\cap I^+$. Thus 
\[
\mbox{$a+1\,$ is a limit point of $\,\,\ol{u}(\psi)\cap I^+$.}
\] 
Then part (4) of 
Claim \ref{cl:th2:8}
gives that $a+1$ is a limit point of $g[\ol{u}(\psi)\cap I^+]$. Since $x\in I=(a,a+r)$, it follows that $a<g(x)$. So there is a point $y\in\ol{u}(\psi)\cap I^+$ with $g(x)$ close to $g(y)$. Since $y\in\ol{u}(\psi)$ we have $U,y\models_u\psi$, and the inductive hypothesis then gives $W,g(y)\models_{gu}\psi$. Since $g(x)$ is close to $g(y)$, we have $W,g(x)\models_{gu}\Di\psi$. 
\end{proof}

By 
Claim \ref{cl:th2:1}
we have $V,0\models_v\vf$. 
Then part (3) of 
Claim \ref{cl:th2:7}
gives that $U,0\models_u\vf$. 
Claim \ref{cl:th2:9} and Claim \ref{cl:th2:10}
established that $U,x\models_u\sigma$ iff $W,g(x)\models_{gu}\sigma$, so, since $g(0)=0$, we have $W,0\models_{gu}\vf$. In particular, $\vf$ is satisfiable in $W$. Since $W=(-d,d)\cap \Q$ and $d$ is the modal depth of $\vf$, it follows from Proposition~\ref{prop:modal-depth} that $\vf$ is satisfiable in $\Q$. Thus 
\Sp

\begin{theorem}\label{thm:less:Qincluded-in-R}
$\Logl(\Q)\subseteq\Logl(\R)$.\end{theorem}


\section{Further properties of the farness logic of $\R$}


In this section we consider further properties of the farness logic of $\R$, that is, the logic $\Logg(\R)$ that uses the modality $\Dg$ which we denote throughout this section as $\Di$. While it is known that this logic is decidable \cite{QTL1999}, we show that it is not finitely axiomatizable and that it does not have the finite model property. Similar studies could be made for the nearness logic $\Logl(\R)$, but we have not done this. 

\subsection{Lack of finite axiomatization}

We show that any axiomatization of $\Logg(\mR)$ contains infinitely many variables occurring in its formulas. In particular, this logic is not finitely axiomatizable. The argument uses a technique from \cite{MaksSkvShehtman}.

For a natural $n$, let $\vf_n$ be the formula 
\[
\Di^2 p_1\con \dots \con \Di^2 p_n \,\imp\, \Di(\Di p_1\con \dots \con \Di p_n).
\]
It is easy to check that this formula is valid in a frame iff it satisfies the first-order property 
\[
\AA x \AA y_1\dots \AA y_n \,(xR^2y_1\con \dots \con  x R^2 y_n \,\imp\, \EE z (xRz \con z R y_1\con \dots \con  z R y_n)).
\]

\begin{prp}\label{prp:nonFA}
Any axiomatization of $\Logg(\mR)$ has infinitely many variables. Consequently, $\Logg(\mR)$ is not finitely axiomatizable.
\end{prp}

\begin{proof} 
Let $\Phi$ be a set of formulas in variables $p_1,\ldots,p_m$ and $\mR\mo\Phi$. Consider the frames $F=(\{0,\ldots, 2^{m}\},R)$, where $xRy$ iff $x\neq y$, and $G=(\{0,\ldots, 2^{m}-1\},R')$, where $x R' y$ iff $x\neq y$ or $x=y=0$. So $F$ is an irreflexive clique, and $G$ is an irreflexive clique with an additional single loop.

Observe that $G$ is a p-morphic image of $(\mR,R_{>1})$: for $1\leq i \leq 2^{m}-1$, map the interval $[2i, 2i+1]$ to $i$, and map the other points in $\mR$ to the reflexive point $0$. It is easy to check that this map is a p-morphism. Hence, $G\mo\Phi$. 
One can see that if a formula  $\vf$ contains at most $m$ distinct variables, then $F\mo \vf$  iff $G\mo \vf$; see \cite[Lemma 7.3]{ShShStoItl05} for details. It follows that $F\mo\Phi$.
Finally, notice that  $\Logg(\mR)$ is not valid in $F$: all formulas $\vf_n$ are valid in $\mR$, while $\vf_{n}$ is not valid in $F$ for $n>2^m$. 
Hence, $\Phi$ is not an axiomatization of $\mR$. 
\end{proof} 

\begin{remark}
The proof above can be easily adapted to show that for any unbounded metric space $X$, its farness logic $\Logg(X)$ is not finitely axiomatizable. 
\end{remark}

\subsection{The finite model property and the farness logic of $\R$}
In this subsection we show that the farness logic $\Logg(\R)$ does not have the \emph{finite model property}, in other words, there is a formula that does not belong to the logic but is valid in all finite models of the logic. 
To show this, we identify a modally-definable property, that of ``small anti-clique overlaps'', that holds in every finite graph that validates the farness logic of $\R$, while the farness graph of $\R$ itself does not have it.

We first provide the needed background. For a binary relation $R$, let $R^*$ be its reflexive and transitive closure. A frame $(X,R)$ is \emph{point-generated}, or \emph{rooted}, if there is $x\in X$ with $X=R^*(x)$ and is {\em pretransitive} if $R^*=\bigcup\{R^\ell\mid\ell\leq k\}$ for some natural number $k$. The following result is folklore, its proof is a small modification of the  Jankov-Fine theorem  \cite[Lemma 3.20]{BDV}.   

\Sp

\begin{proposition}[{\em Jankov-Fine theorem}]\label{prop:JankovFine}
Let $G$ be a pretransitive frame and $F$ be a finite rooted frame. Then $\Log(G)\subseteq \Log(F)$ iff  there is an onto $p$-morphism $G'\toto F$ for some rooted subframe $G'$ of $G$.
\end{proposition}
\Sp

Applying this to the situation at hand, since $\R$ is rooted and each point of $\R$ is a root, it follows that a finite rooted frame validates the farness logic of $\mR$ iff it is a p-morphic image of $\mR$. 
Throughout this sub-section, we assume that every frame $(X,R)$ is symmetric, and $R^2$ is universal, meaning $R^2=X\times X$. An \emph{anti-clique} is a subset $\Sigma$ of $X$ so that $x\,R\,y$ does not hold for any $x,y\in\Sigma$. As noted in (\ref{eqn: max anti-clique}),  
\[
X,R\models_v \Box^2(p\leftrightarrow \neg\Di p) \,\mbox{ iff }\, v(p)\mbox{ is a maximal anti-clique}
\]

\begin{proposition}\label{prop:noFmp:preimageOfMAC} A p-morphic preimage of an anti-clique  is an anti-clique, and that of a maximal anti-clique is a maximal anti-clique.    
\end{proposition}

\begin{proof}
We prove the result for maximal anti-cliques, the result for anti-cliques \mbox{follows} since each anti-clique extends to a maximal one. Let $f:F\toto G$ be an onto $p$-morphism and $\Sigma$ be a maximal anti-clique in $G$. Consider a  valuation in $G$ with $\val(p)=\Sigma$, and a valuation $u$ in $F$ with $\valu(p)=f^{-1}[\Sigma]$. The result follows from the fact that $F,x\mo_\valu \psi$ iff $G,f(x)\mo_\val \psi$ for formulas in the variable $p$.
\end{proof}

\begin{definition}
A frame $(X,R)$ has \emph{small anti-clique overlaps} if any two distinct anti-cliques have at most one element in common. 
\end{definition}
\Sp

To verify that a frame has the small anti-clique overlap property it is enough to verify this for maximal anti-cliques. We wish to express this overlap property in terms of a formula. To this end, let $\psi$ be the following
\[
\Box^2(p\leftrightarrow \neg\Di p)\con \Box^2(q\leftrightarrow \neg\Di q)\con \Di^2(p\con \neg q)
\]
From the discussion above, it is easily seen that for any valuation $v$ on a symmetric frame $(X,R)$ with $R^2$ universal, we have $\ol{v}(\psi)=X$ if $v(p)$ and $v(q)$ are distinct maximal anti-cliques and $\ol{v}(\psi)=\emptyset$ otherwise. Then set $\vf$ to be the following formula 
\[
\psi\wedge\Di^2 (p\con q \con r) \,\imp\, \Box^2(p \con q\imp r) 
\]

\begin{prp}
\label{prop:noFmp:TrivialAnticliqueCondition}
Let $(X,R)$ be a symmetric frame with $R^2$ universal. Then $(X,R)$ has small anti-clique overlaps iff $X,R\models\vf$.
\end{prp}

\begin{proof}
``$\Rightarrow$'' Let $v$ be a valuation for which the premise of $\vf$ holds at some point, and in fact therefore holds everywhere. Then $v(p)$ and $v(q)$ are distinct maximal anti-cliques and $v(r)$ contains a point $x$ of their intersection. Since the frame has small anti-clique overlaps, $v(p)\cap v(q)=\{x\}$, hence $v(p)\cap v(q)\subseteq v(r)$. This shows that the conclusion of $\psi$ also holds everywhere. ``$\Leftarrow$''
Suppose $P,Q$ are distinct maximal anti-cliques and that $x$ belongs to their intersection. Let $v$ be a valuation with $v(p)=P$, $v(q)=Q$ and $v(r)=\{x\}$. Then the premise of $\vf$ holds everywhere under the valuation $v$, and since $\vf$ is valid in the frame, the conclusion must also hold everywhere under the valuation $v$. But this says $P\cap Q\subseteq \{x\}$. 
\end{proof}



%

Maximal anti-cliques in $\R$ are closed unit intervals, so $\R$ clearly does not have small overlaps of anti-cliques, hence $\vf$ is not valid in $\R$.  However, as we will see from the next statement, $\vf$ is valid in all finite p-morphic images of $\mR$. 
\Sp

\begin{lemma} \label{lem:noFmp:dim1}
Each finite p-morphic image of $\R$ has small anti-clique overlaps. 
\end{lemma}

\begin{proof}
Assume $(X,R)$ is a frame and $f:\mR\toto X$ is an onto p-morphism. Then $R$ is symmetric and $R^2$ is universal since these properties hold in $\R$. We assume that $(X,R)$ does not have small anti-clique overlaps, and using this, show that $X$ is infinite. Let $\clA=\{f^{-1}[S]\mid S\subseteq X\}$. Since $f$ is a $p$-morphism, $f^{-1}[\Di S]=\Di f^{-1}[S]$ for each $S\subseteq X$, and it follows that $\clA$ is a modal subalgebra of $\clP(\R)$. 
\vspace{2ex}

\noindent {\bf Claim:} If $[a,b],(c,b]\in \clA$ and $a<b-1<c<b$, then $(b,c+1]\in \clA$.

\begin{proof}[Proof of claim]
$-\Di(c,b]=[b-1,c+1]$, so $(b,c+1]= - (\Di(c,b]\cup [a,b])$.
\end{proof}

Since $(X,R)$ does not have small anti-clique overlaps, it has two distinct \mbox{maximal} anti-cliques that intersect in more than one element. By Proposition~\ref{prop:noFmp:preimageOfMAC} their pre-images under $f$ are distinct maximal anti-cliques and must also intersect in more than one element. Since maximal anti-cliques in $\R$ are closed unit intervals, without loss of generality we assume that they are the intervals $[0,1]$ and $[d,d+1]$ for some $0<d<1$, and since these are preimages of subsets of $X$, they belong to $\clA$. 

We show by induction that the intervals $[0,d+n]$ and $(n,d+n]$ belong to $\clA$ for each natural number $n>0$. This will establish that $\clA$, and hence $X$, are infinite. For the base case, since $[0,1]$ and $[d,d+1]$ belong to $\clA$ and $\clA$ is closed under the Boolean and modal operations, $[0,d+1]$ and $(1,d+1]$ belong to $\clA$. For the inductive step, assume  that $[0,d+n]$ and $(n ,d+n]$ are in $\clA$ for some $n\geq 1$. Since $0<d<1$ we have $0<d+n-1<n<n+d$, so by the claim above, we have 
\[(d+n,n+1]\in\clA.\]
Then, since the inductive hypothesis assumes $[0,d+n]\in\clA$ and $\clA$ is closed under finite unions, we have $[0,n+1]$ belongs to $\clA$. Then, since $[0,n+1], (d+n,n+1]\in\clA$ and $0<n<d+n<n+1$, the claim above yields 
\[
(n+1,d+n+1]\in\clA.
\]
Since $[0,n+1]\in\clA$ we have $[0,d+n+1],(n+1,d+n+1]\in\clA$ as required. 
\end{proof}

\begin{theorem}\label{thm:noFMP}
$\Logg(\R)$ does not have the finite model property. 
\end{theorem}

\begin{proof}
Let $L$ be the logic $\Logg(\R)$ and suppose that $F$ is a finite $L$-frame. It is well-known the logic of $F$ is the logic of the set $\clK$  of rooted subframes of $F$, all of which are clearly also $L$-frames. If $(X,R)\in\clK$, then by Proposition~\ref{prop:JankovFine} $(X,R)$ is a p-morphic image of $\R$. Then $R$ is symmetric and $R^2$ is an equivalence since these properties can be expressed by formulas that are valid in $\R$. Since $(X,R)$ is rooted and $R^2$ is an equivalence, it follows that $R^2$ is universal. So by Lemma~\ref{lem:noFmp:dim1} it has small anti-clique overlaps. Then by Proposition~\ref{prop:noFmp:TrivialAnticliqueCondition} $X,R\models\phi$. So the formula $\vf$ is valid in $\clK$ and hence in $F$. So $\vf$ is valid in all finite $L$-frames but it does not belong to $\Logg(\R)$. 
\end{proof}

\section{Conclusions}

We have considered Euclidean spaces $\R^n$ equipped with modalities $\Dg,\Dl$ and $\De$ and have shown that these languages a surprisingly expressive. In particular, in each modality the logics for $\R^n$ for $n\in\N$ are all distinct. 
\Sp

\noindent {\bf Comparing farness logics}

\begin{itemize}
\itemindent = 5ex
\item[(1)] The logics $\Logg(\R^n)$ are all distinct.  
\end{itemize}

\noindent {\bf Comparing nearness logics} 

\begin{itemize}
\itemindent = 5ex
\item[(1)]  the logics $\Logl(\R^n)$ form a strictly decreasing chain. 
\end{itemize}

\noindent {\bf Comparing constant distance logics}

\begin{itemize}
\itemindent = 5ex
\item[(1)] $\Loge(\R^n)\not\subseteq\Loge(\R^m)\,$ for each $n<m$    
\item[(2)] $\Loge(\R^m)\not\subseteq\Loge(\R^n)\,$ for $m$ sufficiently larger than $n$
\item[(3)] $\Loge(\R)$ and $\Loge(\R^2)$ are incomparable to all other logics
\end{itemize}

\noindent It follows that the poset of constant distance logics contains an infinite anti-chain. 
\Sp

\noindent {\bf Comparing  logics of $\Q$ and $\R$}

\begin{itemize}
\itemindent = 5ex
\item[(1)] $\Logg(\Q)$ is strictly contained in $\Logg(\R)$  
\item[(2)] $\Logl(\Q)$ is strictly contained in $\Logl(\R)$
\item[(3)] $\Loge(\Q)$ is equal to $\Loge(\R)$
\end{itemize}

\noindent {\bf Negative results on the farness logic of $\R$}

\begin{itemize}
\itemindent = 5ex
\item[(1)] $\Logg(\R)$ cannot be axiomatized with finitely many variables  
\item[(2)] $\Logg(\R)$ does not have the finite model property
\end{itemize}

\noindent {\bf Directions for further study} There are unsolved problems in the directions we have considered and other directions worthy of study.
\Sp

\begin{problem}
What is the order structure of the logics $\Logg(\R^n)$?   
\end{problem}
\Sp

\begin{problem}
Do the nearness logics $\Logl(\R^n)$ form an anti-chain?   
\end{problem}
\Sp

\begin{problem}
Study the logics of $\Q^n$ and $\R^n$ in the modalities $\Dg,\Dl$ and $\De$. 
\end{problem}
\Sp

\begin{problem}
Study the logics of the Cantor set and $p$-adic numbers in various modalities. 
\end{problem}
\Sp

It follows from \cite{QTL1999} that the logics of farness and nearness on the real line are decidable. We also know that in a richer language (with a family $\Di_{<r}$, $\Di_{\leq r}$ of nearness  operators  and the topological closure modality), the logic of the real plane is undecidable \cite{Wolter2005}. 

\Sp

\begin{problem}
Are the logics $\Logg(\R^n)$, $\Logl(\R^n)$ and $\Loge(\R^n)$ decidable, or at least recursively axiomatizable, in dimension $n>1$?
\end{problem}
\Sp

In \cite{Kutz-Sturm-Suzuki-Wolter-Zakharyaschev2002,Wolter2005,Kutz2007} axiomatizations of the class of all metric spaces are given for many powerful distance languages containing various families of modalities. In particular, an axiomatization of the farness logic of the class of all metric spaces follows from \cite{Kutz2007}.  
The axiomatization problem for specific classes of metric spaces in various languages is of interest \cite{Wolter2005,KuruczWZ05}. We have obtained axiomatizations for the farness logics of the class of unbounded metric spaces and the class of ultrametric spaces and will provide these in a subsequent paper. 

We have shown in Section~5 that the farness logic of the reals cannot be finitely axiomatized. We have numerous formulas that hold in the nearness logic $\Logl(\R)$ but not in 
the nearness logics of all metric spaces, or even in $\R^n$ for higher dimensions. We anticipate that the nearness logic of $\R$ is not finitely axiomatizable.

Finally, there are many possibilities in terms of considering additional Euclidean logics, such as the logics $\Log_{\scaleto{\,\leq}{3.5pt}1}(\R^n)$ and $\Log_{\scaleto{\,\geq}{3.5pt}1}(\R^n)$ obtained from the modalities 
$\Di_{\leq 1}$ and $\Di_{\geq 1}$. One could also consider Euclidean logics where nearness is interpreted as strictly positive distance of at most 1, and so forth. These would be interesting topics for future study. 






\bigskip 

\bmhead{Acknowledgements}
This work was supported by NSF Grant DMS - 2231414. We are grateful to Andrzej Ehrenfeucht and Michael Zakharyaschev  for valuable discussions. Finally, we thank the referees for carefully reading the manuscript and for their many helpful suggestions. 

\bibliographystyle{unsrt}
\bibliography{metric}

\end{document}